\documentclass{amsart}

\begin{document}

\renewcommand{\thesubsection}{\arabic{subsection}}

\newcommand{\nc}{\newcommand}

\nc{\g}{\mathfrak g}
\nc{\n}{\mathfrak n} \nc{\opn}{\overline{\n}}
\nc{\h}{\mathfrak h}
\nc{\ba}{\mathfrak b}
\nc{\Ug}{U(\g)} \nc{\Uh}{U(\h)} \nc{\Un}{U(\n)}
\nc{\Uopn}{U(\opn)}\nc{\Ub}{U(\ba)} \nc{\p}{\mathfrak p}
\nc{\z}{\mathfrak z}
\nc{\m}{\mathfrak m}
\nc{\ka}{\mathfrak k}
\nc{\opk}{\overline{\ka}}
\nc{\opb}{\overline{\ba}}
\nc{\e}{{\epsilon}}
\nc{\ke}{{\bf k}_\e}
\nc{\Hk}{{\rm Hk}^{\gr}(A,A_0,\e )}
\nc{\gr}{\bullet}
\nc{\ra}{\rightarrow}
\nc{\Alm}{A-{\rm mod}}
\nc{\DAl}{{D}^-(A)}
\nc{\HA}{{\rm Hom}_A}

\newtheorem{theorem}{Theorem}{}
\newtheorem{lemma}[theorem]{Lemma}{}
\newtheorem{corollary}[theorem]{Corollary}{}
\newtheorem{conjecture}[theorem]{Conjecture}{}
\newtheorem{proposition}[theorem]{Proposition}{}
\newtheorem{axiom}{Axiom}{}
\newtheorem{remark}{Remark}{}
\newtheorem{example}{Example}{}
\newtheorem{exercise}{Exercise}{}
\newtheorem{definition}{Definition}{}

\title{The geometric meaning of Zhelobenko operators}

\author{A. Sevostyanov}

\address{ Institute of Mathematics,
University of Aberdeen \\ Aberdeen AB24 3UE, United Kingdom \\e-mail: seva@maths.abdn.ac.uk}

\begin{abstract}
Let $\g$ be the complex semisimple Lie algebra associated to a complex semisimple algebraic group $G$, $\ba$ a Borel subalgebra of $\g$, $\h\subset \ba$ the Cartan sublagebra and $N\subset G$ the unipotent subgroup corresponding to the nilradical $\n\subset \ba$.
We show that the explicit formula for the extremal projection operator for $\g$ obtained by Asherova, Smirnov and Tolstoy and similar formulas for Zhelobenko operators are related to the existence of a birational equivalence $N\times \h \rightarrow \ba$ given by the restriction of the adjoint action. Simple geometric proofs of  formulas for the ``classical'' counterparts of the extremal projection operator and of Zhelobenko operators are also obtained.
\end{abstract}


\maketitle

\section*{Introduction}

Let $\g$ be a complex finite--dimensional semisimple Lie algebra, $\n\subset \g$ its maximal nilpotent subalgebra.
Extremal projection operators are projection operators onto the subspaces of $\n$--invariants in certain $\g$--modules the action of $\n$ on which is locally nilpotent. The first example of such operators for $\g=\mathfrak{sl}_2$ was explicitly constructed in \cite{L}. In papers \cite{AST1,AST2,AST3} the results of \cite{L} were generalized to the case of arbitrary complex semisimple Lie algebras, and explicit formulas for extremal projection operators were obtained. Later, using a certain completion of an extension of the universal enveloping algebra of $\g$, Zhelobenko observed in \cite{Z1} that the existence of extremal projection operators is an almost trivial fact. In \cite{Z1} he also introduced a family of operators which are analogues to extremal projection operators. These operators are called now Zhelobenko operators.

Extremal projection operators and Zhelobenko operators appear in various contexts. Besides the theory of $\g$--modules one should mention the theory of crystal bases where similar operators are known under the name of Kashiwara operators (see \cite{K1,K2}), Mickelsson algebras (see \cite{Z1}) and dynamical Weyl group elements which are examples of Zhelobenko operators (see \cite{EV}). An extended review of applications of extremal projection operators and of Zhelobenko operators can be found in \cite{T}, and book \cite{Z2} contains a detailed exposition of their applications. A more recent application to the structure theory of complex semisimple Lie algebras can be found in \cite{KNV}.

Despite of the fact that the explicit formula for extremal projection operators was obtained in  \cite{AST1} more that forty years ago and Zhelobenko operators are also defined by explicit formulas it seems that the meaning of those formulas is still not clear. Actually the proof of the explicit formula for extremal projection operators given in \cite{AST3} is reduced to case by case analysis for Lie algebras of rank 2, and the remarkable properties of Zhelobenko operators are derived in \cite{Z3} by applying a series of complicated technical arguments which are not transparent.

The aim of this short note is to fill in this gap. Following Zhelobenko, we consider the extremal projection operator for a universal Verma module and its classical analogue, the space of functions on the Borel subalgebra $\ba$ containing $\n$. The point is that the natural classical analogue of the extremal projection operator is a projection operator onto the subspace of $\n$--invariant functions on $\ba$, where the action is induced by the action of $\n$ on $\ba$ by commutators. A formula for that operator, which is similar to the formula derived in \cite{AST3}, can be easily obtained by elementary geometric methods. Zhelobenko operators can be treated in a similar way. Moreover, classical analogues of extremal projection operators are particular examples of the general construction presented in the next section.

\section{Projection operators associated to free group actions on manifolds}

Let $G$ be a Lie group (algebraic group), and $M$ a manifold (algebraic variety) equipped with a Lie group (regular) action of $G$. Assume that there exists a cross--section $X\subset M$ for this action, and $M=G\times X$ as a $G$--manifold (variety), so that the map
\begin{equation}\label{iso}
G\times X \rightarrow M,~(g,x)\mapsto gx,~g\in G,x\in X
\end{equation}
is an isomorphism (an isomorphism of manifolds, an isomorphism of varieties or a birational equivalence).

Denote by $F(M)$ the space of smooth (regular or rational) functions on $M$ depending on the context. The group $G$ naturally acts on $F(M)$. Let $F(M)^G$ be the subspace of $G$--invariant elements of $F(M)$.
Define the projection operator $\overline{P}:F(M)\rightarrow F(M)^G$ as follows
\begin{equation}\label{P}
(\overline{P}f)(gx)=f(x)=f(g^{-1}gx)=(gf)(gx),~g\in G,x\in X.
\end{equation}
If for $y\in M$ we denote by $g(y)\in G$ the unique element such that $y=g(y)x(y)$ for a unique $x(y)\in X$ then
\begin{equation}\label{P1}
(\overline{P}f)(y)=(g(y)f)(y).
\end{equation}
We call the operator $\overline{P}$ the projection operator corresponding to isomorphism (\ref{iso}).

If every element of $G$ can be uniquely represented as a product of elements from subgroups $G_1, \ldots, G_N$, i.e. $G=G_1\cdot \ldots \cdot G_N$, then the operator $P$ can be expressed as a composition of operators $\overline{P}_i$,
\begin{equation}\label{Pi}
(\overline{P}_if)(y)=(g_i(y)f)(y),~g(y)=g_1(y)\ldots g_N(y),~g_i(y)\in G_i,
\end{equation}
\begin{equation}\label{PC}
(\overline{P}f)(y)=(\overline{P}_1\ldots \overline{P}_Nf)(y).
\end{equation}

\section{Quasi--classical versions of extremal projection operators}

Now let $\g$ be the complex semisimple Lie algebra associated to a complex semisimple algebraic group $G$, $\ba$ a Borel subalgebra of $\g$, $\h\subset \ba$ the Cartan sublagebra and $N\subset G$ the unipotent subgroup corresponding to the nilradical $\n\subset \ba$. Let $\Delta$ be the set of roots associated to the pair $(\g,\h)$, $\Delta_+$ the set of positive roots associated to the pair $(\g,\ba)$, and $\h'=\{h\in \h:\alpha(h)\neq 0~\forall \alpha \in \Delta_+\}$ the regular part of $\h$. Denote by $\ba'\subset \ba$ the open (in the Zariski topology) subset $\ba'=\h'+\n$.

Consider the action of $N$ on $\ba$ induced by the adjoint action. Obviously this action induces an action of $N$ on $\ba'$. It turns out that the multiplicative formula for extremal projection operators obtained in \cite{AST1,AST2,AST3} is related to the property of the action map $N\times \h' \rightarrow \ba'$ formulated in the following lemma.

\begin{lemma}
The action map
\begin{equation}\label{iso1}
N\times \h' \rightarrow \ba'
\end{equation}
induced by the adjoint action is an isomorphism of manifolds, and hence this map gives rise to a birational equivalence
\begin{equation}\label{isob}
N\times \h \rightarrow \ba.
\end{equation}
\end{lemma}

\begin{proof}
Indeed, fix a normal ordering $\beta_1, \ldots ,\beta_N$ of $\Delta_+$, i.e. an ordering of $\Delta_+$ in which all simple roots are placed in an arbitrary way and if $\gamma=\alpha+\beta$, $\alpha, \beta, \gamma \in \Delta_+$ then $\gamma$ is situated between $\alpha$ and $\beta$. Let $e_\alpha$, $\alpha \in \Delta$ be root vectors in $\g$ and $h_\alpha\in \h$ are coroots normalized in such a way that $[h_\alpha,e_{\pm\alpha}]=\pm 2e_\alpha$, $[e_\alpha,e_{-\alpha}]=h_\alpha$, $\alpha \in \Delta_+$.

Every element $y\in \ba'$ can be uniquely written in the form
\begin{equation}\label{y}
y=h+\sum_{\alpha\in \Delta_+}c_\alpha e_\alpha,~h\in \h',c_\alpha \in \mathbb{C}.
\end{equation}
Now recall that $[e_\alpha,e_\beta]\in \g_{\alpha+\beta}$, where $\g_{\alpha+\beta}\subset \g$ is the root subspace corresponding to the root $\alpha+\beta$.  From the definition of the normal ordering of $\Delta_+$ it also follows that if for $\alpha\in \{\beta_1, \ldots, \beta_N\}$ we have $\beta_N+\alpha\in \Delta$ then $\beta_N+\alpha\in \{\beta_1, \ldots, \beta_{N-1}\}$. Therefore using the commutation relation $\frac{c_{\beta_N}}{\beta_N(h)}[e_{\beta_N},h]=-c_{\beta_N}e_{\beta_N}$ and the formula ${\rm Ad} e^x=e^{{\rm ad}~ x}$, where for $x\in \g$ $e^x$ stands for the exponential mapping $\exp:\g\rightarrow G$ applied to $x$, one obtains
$$
y^1={\rm Ad} e^{\frac{c_{\beta_N}}{\beta_N(h)}e_{\beta_N}}y=h+\sum_{\alpha<\beta_N}c_\alpha^1 e_\alpha,~h\in \h',
$$
where the coefficients $c_\alpha^1 \in \mathbb{C}$, $\alpha<\beta_N$ are uniquely defined for given $y$.

We can proceed in the same way
to obtain
$$
y^2={\rm Ad} e^{\frac{c_{\beta_{N-1}}^1}{\beta_{N-1}(h)}e_{\beta_{N-1}}}y^1=h+\sum_{\alpha<\beta_{N-1}}c_\alpha^2 e_\alpha,~h\in \h',c_\alpha^2 \in \mathbb{C}.
$$

Now a simple induction gives the following representation for $y\in \ba'$,
$$
{\rm Ad} e^{\frac{c_{\beta_{1}}^{N-1}}{\beta_{1}(h)}e_{\beta_1}}\ldots{\rm Ad} e^{\frac{c_{\beta_{N-1}}^1}{\beta_{N-1}(h)}e_{\beta_{N-1}}}{\rm Ad} e^{\frac{c_{\beta_N}}{\beta_N(h)}e_{\beta_N}}y=h,
$$
or
\begin{equation}\label{isof}
{\rm Ad} e^{-\frac{c_{\beta_N}}{\beta_N(h)}e_{\beta_N}}{\rm Ad}  e^{-\frac{c_{\beta_{N-1}}^1}{\beta_{N-1}(h)}e_{\beta_{N-1}}}\ldots {\rm Ad}e^{-\frac{c_{\beta_{1}}^{N-1}}{\beta_{1}(h)}e_{\beta_1}}h=y,
\end{equation}
where the coefficients $c_{\beta_{i}}^{N-i}$ are uniquely defined by induction from the relations
\begin{equation}\label{ind}
{\rm Ad} e^{\frac{c_{\beta_{N-i+1}}^{i-1}}{\beta_{N-i+1}(h)}e_{\beta_{N-i+1}}}\ldots {\rm Ad} e^{\frac{c_{\beta_N}}{\beta_N(h)}e_{\beta_N}}y=h+\sum_{\alpha<\beta_{N-i+1}}c_\alpha^{i} e_\alpha.
\end{equation}
This establishes isomorphism (\ref{iso1}).

\end{proof}

Now following (\ref{Pi}) we define operators $P_\alpha:F(\ba')\rightarrow F(\ba')$, $\alpha\in \Delta_+$ as follows
\begin{equation}\label{Pa}
(P_\alpha f)(y)=f({\rm Ad} e^{\frac{c_{\alpha}}{\alpha(h)}e_{\alpha}}y)=(e^{-\frac{c_{\alpha}}{\alpha(h)}e_{\alpha}}f)(y),
\end{equation}
where $y$ is given by (\ref{y}).

Introduce functions $\widehat{h}_\alpha$, $\widehat{e}_\alpha$, $\alpha\in \Delta$ on $\g$ by $\widehat{e}_\alpha(y)=c_{-\alpha}$, $\widehat{h}_\alpha(y)=\alpha(h)$ where $y=h+\sum_{\alpha\in \Delta}c_\alpha e_\alpha,~h\in \h,c_\alpha \in \mathbb{C}$.

For $\alpha\in \Delta_+$ we denote the restrictions of $\widehat{e}_{-\alpha}$ and $\widehat{h}_\alpha$ to $\ba'$ by the same letters.

Using these functions one can also rewrite formula (\ref{Pa}) as follows
\begin{equation}\label{Ps}
(P_\alpha f)(y)=\sum_{n=0}^\infty \frac{(-1)^n}{n!}\widehat{h}_\alpha^{-n}(y)\widehat{e}_{-\alpha}^n(y) (e_\alpha^n f)(y),
\end{equation}
where $e_\alpha^n f$ stands for the Lie algebra action of the root vector $e_\alpha$ induced by the Lie group action of $N$ on $\ba'$.

According to (\ref{PC}) the projection operator $P:F(\ba')\rightarrow F(\ba')^N$ corresponding to isomorphism (\ref{iso1}) can be represented in the form
\begin{equation}\label{pfact}
P=P_{\beta_N}\ldots P_{\beta_1}.
\end{equation}

Here $F(\ba')$ is the algebra of $C^\infty$--functions on $\ba'$ or, recalling that (\ref{isob}) is a birational equivalence, the algebra of rational functions on $\ba$. Let $\mathbb{C}[\h]'$ be the algebra of rational functions on $\h$ denominators of which are products of linear factors of the form $\alpha(h)$, $\alpha\in \Delta_+$, $h\in \h$. Looking at formula (\ref{Ps}) one can immediately deduce that $P$ is in fact defined in a smaller algebra $\mathbb{C}[\ba]'=\mathbb{C}[\ba]\otimes_{\mathbb{C}[\h]}\mathbb{C}[\h]'$, the localization of the algebra of regular functions $\mathbb{C}[\ba]$ on $\ba$ on $\mathbb{C}[\h]'$.

One can also rewrite formula (\ref{Ps}) in terms of the Kirillov-Kostant Poisson bracket on $\g^*$. Indeed, one can identify $\g^*$ and $\g$ using the Killing form on $\g$ and equip the algebra of regular functions $\mathbb{C}[\g]$ with the Kirillov-Kostant Poisson bracket. This bracket naturally extends to the localization $\mathbb{C}[\g]'=\mathbb{C}[\g]\otimes_{\mathbb{C}[\h]}\mathbb{C}[\h]'$.

Now assume that the Killing form is normalized in such a way that $\{\widehat{e}_\alpha,\widehat{e}_{-\alpha}\}=\widehat{h}_\alpha$, $\{\widehat{h}_\alpha,\widehat{e}_{\pm\alpha}\}=\pm 2\widehat{e}_{\pm \alpha}$, $\alpha \in \Delta_+$.

Note that the algebra $\mathbb{C}[\ba]'$ can be naturally identified with the quotient $\mathbb{C}[\g]'/I$, where $I$ is the ideal in $\mathbb{C}[\g]'$ which consists of elements vanishing on $\ba$. Since $\g=\bigoplus_{\alpha \in \Delta}\g_\alpha \oplus \h$ is a direct sum of vector spaces and  $\ba=\bigoplus_{\alpha \in \Delta_+}\g_\alpha \oplus \h$ the ideal $I$ is
generated by the functions $\widehat{e}_\alpha$, $\alpha\in \Delta_+$.

The action of $N$ on $\g$ induced by the adjoint action of $G$ becomes Hamiltonian with respect to the Kirillov--Kostant Poisson structure, and for any function $\varphi$ on $\g$ we have
${e}_\alpha\varphi=\{\widehat{e}_\alpha,\varphi\}$.
Using this observation formula (\ref{Ps}) can be rewritten in the form
\begin{equation}\label{Pp}
(P_\alpha f)(y)=\sum_{n=0}^\infty \frac{(-1)^n}{n!}\widehat{h}_\alpha^{-n}(y)\widehat{e}_{-\alpha}^n(y) \{\widehat{e}_\alpha, \widetilde{f}\}^n(y)~{\rm (mod}~I{\rm )},
\end{equation}
where $\widetilde{f}\in \mathbb{C}[\g]'$ is any representative of $f\in \mathbb{C}[\ba]'=\mathbb{C}[\g]'/I$ in $\mathbb{C}[\g]'$, and
$$
\{e_\alpha \widetilde{f}\}^n=\underbrace{\{\widehat{e}_\alpha,\{\widehat{e}_\alpha, \ldots,\{\widehat{e}_\alpha,\widetilde{f}\}\ldots\}\}}_{\rm n~times}.
$$
Note that the ideal $I$ is invariant under taking Poisson brackets with functions $\widehat{e}_\alpha$, $\alpha\in \Delta_+$, and hence the right hand side of (\ref{Pp})  does not depend on the choice of the representative $\widetilde{f}$.

We can summarize the above discussion in the following proposition.
\begin{proposition}\label{1}
Let $G$ be a complex semisimple algebraic group with Lie algebra $\g$, $\h$ a Cartan subalgebra of $\g$, $\ba$ the Borel subalgebra of the Lie algebra $\g$ containing $\h$, $\n\subset \ba$ the nilradical of $\ba$, and $N$ the maximal unipotent subgroup $N\subset G$ corresponding to $\n$. Let $\mathbb{C}[\ba]'=\mathbb{C}[\ba]\otimes_{\mathbb{C}[\h]}\mathbb{C}[\h]'$ be the localization of the algebra of regular functions $\mathbb{C}[\ba]$ on $\ba$ on the algebra of rational functions $\mathbb{C}[\h]'$ on $\h$ denominators of which are products of linear factors of the form $\alpha(h)$, $\alpha\in \Delta_+$, $h\in \h$.

Let $N\times \ba \rightarrow \ba$ be the action of $N$ on $\ba$ induced by the adjoint action of $G$. Then the map
$$
N\times \h \rightarrow \ba
$$
is a birational equivalence, and the corresponding projection operator $P:\mathbb{C}[\ba]'\rightarrow \mathbb{C}[\ba]'^N$ can be written in the form
$$
P=P_{\beta_N}\ldots P_{\beta_1},
$$
where $\beta_1, \ldots ,\beta_N$ is a normal ordering of the system of positive roots $\Delta_+$ of the pair $(\g,\ba)$, and for $\alpha\in \Delta_+$ the operators $P_\alpha$ are given by formulas (\ref{Ps}) or (\ref{Pp}).

The kernel of the projection operator $P$ is the ideal in the algebra $\mathbb{C}[\ba]'$ generated by the elements $\widehat{e}_{-\alpha}$, $\alpha\in \Delta_+$.
\end{proposition}

One can also consider partial projection operators $P_{\geq k}=P_{\beta_N}\ldots P_{\beta_k}$ acting on $\mathbb{C}[\ba]'$. Arguments similar to those given above show that the image of $P_{\geq k}$ consists of elements of $\mathbb{C}[\ba]'$ which are invariant with respect to the action of the subgroup $N_{\geq k}\subset N$ the Lie algebra of which is generated by the root vectors $e_\alpha$, $\alpha\geq \beta_k$ with respect to the fixed normal ordering of $\Delta_+$,
\begin{equation}\label{Ppart}
P_{\geq k}:\mathbb{C}[\ba]'\rightarrow \mathbb{C}[\ba]'^{N_{\geq k}}.
\end{equation}

Now we compare the operator $P$ with extremal projection operators. Let $\Ug$ be the universal enveloping algebra of $\g$, $\Ug'=\Ug \otimes_{\Uh}D(\h)$ the localization of $\Ug$ on the algebra of fractions $D(\h)$ of the universal enveloping algebra $\Uh$ of $\h$ denominators of which are products of factors of the form $h_\alpha+m$, $\alpha\in \Delta_+$, $m\in \mathbb{Z}$. Consider the left $\g$--module $V=\Ug'/\Ug'\n$. $V$ is a ``noncommutative'' analogue of the space $\mathbb{C}[\ba]'$. Since the action of the Lie algebra $\n$  on $V$ is locally nilpotent the operators
\begin{equation}\label{Pq}
p_\alpha(t)=\sum_{n=0}^\infty \frac{(-1)^n}{n!}f_{\alpha,n}^{-1}(t){e}_{-\alpha}^n e_\alpha^n,~f_{\alpha,n}(t)=\prod_{j=1}^n({h}_\alpha+t+j),~t\in \mathbb{C},~\alpha\in \Delta_+
\end{equation}
are well defined on $V$.
\begin{proposition}{\bf (\cite{AST3}, Main Theorem)}\label{2}
The operator
$$
p=p_{\beta_N}(\rho(h_{\beta_N}))\ldots p_{\beta_1}(\rho(h_{\beta_1})),
$$
where $\rho=\frac{1}{2}\sum_{\alpha\in \Delta_+}\alpha$, is the projection operator onto the subspace
$$
V^{\n}=\{v\in V:~xv=0~\forall~x\in \n\}
$$
with the kernel $\opn V$, where $\opn$ is the maximal nilpotent subalgebra of $\g$ opposite to $\n$.
\end{proposition}

The operator $p$ is called the extremal projection operator for the Lie algebra $\g$.

The operators $p_\alpha(t)$ can also be rewritten in terms of iterated commutators as follows
\begin{equation}\label{Pq1}
p_\alpha(t)v=\sum_{n=0}^\infty \frac{(-1)^n}{n!}f_{\alpha,n}^{-1}(t){e}_{-\alpha}^n ({\rm ad}~e_\alpha)^n(\widetilde{v})v_0,
\end{equation}
where $\widetilde{v}$ is any representative of $v$ in $\Ug'$, and $v_0$ is the image of $1\in \Ug'$ under the natural projection $\Ug'\rightarrow V=\Ug'/\Ug'\n$.

Formulas (\ref{Pq}) and (\ref{Pq1}) are natural analogues of (\ref{Ps}) and (\ref{Pp}) for the module $V$, and the projection operator $p$ is a counterpart of $P$ for $V$. Note that $\Ug$ can be obtained by quantizing the Kirillov--Kostant Poisson bracket on $\mathbb{C}[\g]$, and hence Proposition \ref{2} can be regarded as a quantum version of Proposition \ref{1}. Note that after quantization the functions $\widehat{e}_\alpha$, $\widehat{h}_\alpha$ become the root vectors $e_\alpha$ and the Cartan subalgebra generators $h_\alpha$. Integer shifts in the denominators in formula (\ref{Pq}) are purely quantum corrections.

\section{Quasi--classical versions of extremal projection operators and the cohomology of the nilradical}

Now we would like to mention a relation of birational equivalence (\ref{isob}) to the problem of calculation of the cohomology ${\rm Ext}^\bullet_{U(\g)}(U(\g)\otimes_{U(\n)}\mathbb{C}_0,U(\g)\otimes_{U(\n)}\mathbb{C}_0)$, where $\mathbb{C}_0$ is the trivial representation of $U(\n)$. This problem was formulated by M. Duflo in \cite{duf}. The importance of this cohomology is due to the fact that ${\rm Ext}^\bullet_{U(\g)}(U(\g)\otimes_{U(\n)}\mathbb{C}_0,U(\g)\otimes_{U(\n)}\mathbb{C}_0)$ is in fact an algebra acting on the spaces ${\rm Ext}_{U(\n)}^\bullet({\mathbb C}_{0},V)$ for left $U({\mathfrak g})$--modules $V$ (see \cite{7,11}). This algebra is an example of the quantum BRST cohomology (see \cite{KSt}).

Note that by Frobenius reciprocity one has
\begin{eqnarray*}{\rm Ext}^\bullet_{U(\g)}(U(\g)\otimes_{U(\n)}\mathbb{C}_0,U(\g)\otimes_{U(\n)}\mathbb{C}_0)={\rm Ext}^\bullet_{U(\n)}(\mathbb{C}_0,U(\g)\otimes_{U(\n)}\mathbb{C}_0)= \\ ={\rm Ext}^\bullet_{U(\n)}(\mathbb{C}_0,U(\g)/U(\g)\n),
\end{eqnarray*}
and one can consider the corresponding ``classical'' BRST cohomology which is the cohomology ${\rm Ext}^\bullet_{U(\n)}(\mathbb{C}_0,\mathbb{C}[\ba])$ of the graded $U(\n)$--module $\mathbb{C}[\ba]$ associated to the filtered $U(\n)$--module $U(\g)/U(\g)\n$, the filtration being induced by the canonical filtration on $U(\g)$. Here the action of $U(\n)$ on $\mathbb{C}[\ba]$ is induced by the adjoint action of $\n$ on $\ba$. The cohomology ${\rm Ext}^\bullet_{U(\n)}(\mathbb{C}_0,\mathbb{C}[\ba])$ is an example of the classical BRST cohomology. In particular, ${\rm Ext}^\bullet_{U(\n)}(\mathbb{C}_0,\mathbb{C}[\ba])$ is naturally a Poisson algebra (see \cite{KSt}).

It is easy to show that ${\rm Ext}^0_{U(\n)}(\mathbb{C}_0,U(\g)/U(\g)\n)=U(\h)$ and ${\rm Ext}^0_{U(\n)}(\mathbb{C}_0,\mathbb{C}[\ba])=\mathbb{C}[\h]$. But in both cases positive graded components of the cohomology are nontrivial.

Now consider the cohomology ${\rm Ext}^\bullet_{U(\n)}(\mathbb{C}_0,\mathbb{C}[\ba]')$. Formula (\ref{isof}) for birational equivalence (\ref{isob}) implies that we have an $U(\n)$--module isomorphism $\mathbb{C}[\ba]'\simeq \mathbb{C}[N]\otimes \mathbb{C}[\h]'$, where the action of $U(\n)$ on $\mathbb{C}[N]\otimes \mathbb{C}[\h]'$ is induced by the action of $\n$ on $N$ by right invariant vector fields. Thus
\begin{eqnarray*}
{\rm Ext}^\bullet_{U(\n)}(\mathbb{C}_0,\mathbb{C}[\ba]')={\rm Ext}^\bullet_{U(\n)}(\mathbb{C}_0,\mathbb{C}[N]\otimes \mathbb{C}[\h]')={\rm Ext}^\bullet_{U(\n)}(\mathbb{C}_0,\mathbb{C}[N])\otimes \mathbb{C}[\h]'= \\ =H_{dR}^\bullet(N)\otimes \mathbb{C}[\h]'=\mathbb{C}[\h]'
\end{eqnarray*}
since ${\rm Ext}^\bullet_{U(\n)}(\mathbb{C}_0,\mathbb{C}[N])=H_{dR}^\bullet(N)$, where $H_{dR}^\bullet(N)$ is the de Rham cohomology of $N$, and $H_{dR}^\bullet(N)=\mathbb{C}$ because $N$ is unipotent.
In particular, ${\rm Ext}^0_{U(\n)}(\mathbb{C}_0,\mathbb{C}[\ba]')=\mathbb{C}[\h]'$, and the positive graded components of the cohomology ${\rm Ext}^\bullet_{U(\n)}(\mathbb{C}_0,\mathbb{C}[\ba]')$ are trivial. This indicates that nontrivial cohomology classes of positive degrees in ${\rm Ext}^\bullet_{U(\n)}(\mathbb{C}_0,\mathbb{C}[\ba])$ exist because the space $\mathbb{C}[\ba]$ does not contain some elements with $\mathbb{C}[\h]$--valued denominators which are products of linear factors of the form $\alpha(h)$, $\alpha\in \Delta_+$, $h\in \h$, and all nontrivial cohomology classes of positive degrees in ${\rm Ext}^\bullet_{U(\n)}(\mathbb{C}_0,\mathbb{C}[\ba])$ become trivial in ${\rm Ext}^\bullet_{U(\n)}(\mathbb{C}_0,\mathbb{C}[\ba]')$. Similar fact should hold for the positive degree cohomology classes from the space ${\rm Ext}^\bullet_{U(\n)}(\mathbb{C}_0,U(\g)/U(\g)\n)$.
This deserves further investigation in the framework of algebraic geometry.

\section{Quasi--classical analogues of Zhelobenko operators}

Now we would like to present an analogue of Proposition \ref{1} for the so--called Zhelobenko operators.
In \cite{Z1} Zhelobenko considered the following formal series
\begin{eqnarray}\label{qa}
q_\alpha' (\widetilde{x})=\sum_{n=0}^\infty \frac{(-1)^n}{n!}({\rm ad}~e_\alpha)^n(\widetilde{x}){e}_{-\alpha}^ng_{\alpha,n}^{-1},~g_{\alpha,n}= \\
=\prod_{j=1}^n({h}_\alpha+1-j),~\alpha \in \Delta_+,~\widetilde{x}\in \Ug'. \nonumber
\end{eqnarray}

Let $W$ be the Weyl group of the pair $(\g,\h)$. Denote by $s_\alpha$ the reflection with respect to the root $\alpha \in \Delta$.
One can consider the formal compositions
\begin{equation}\label{qw}
q_w'(\widetilde{x})=q_{\beta_N}'(\ldots (q_{\beta_k}'(\widetilde{x}))\ldots ),~\widetilde{x}\in \Ug',
\end{equation}
where $\beta_1, \ldots ,\beta_N$ is a normal ordering of $\Delta_+$, and
\begin{equation}\label{w}
w=s_{\beta_k}\ldots s_{\beta_N}.
\end{equation}
Note that according to the results of \S 3 in \cite{Z3} every element $w\in W$ can be represented in form (\ref{w}),
and $\Delta_w=\{ \alpha\in \Delta_+:~w^{-1}\alpha<0 \}=\{ \beta_k,\ldots ,\beta_N \}$.

For $w\in W$ we shall also consider the $\Ug'$--modules $V_w=\Ug'/\Ug'\n^w$, where $\n^w$ is the Lie subalgebra of $\g$ generated by the root vectors $e_{w\alpha}$, $\alpha \in \Delta_+$.
\begin{proposition}{\bf (\cite{Z3}, Proposition 1, Theorem 2)}\label{3}
For any element $w\in W$ the composition defined by the right hand side of formula (\ref{qw}) is well defined as an operator
$$
q_w:V_w\rightarrow V, ~x\mapsto q_w'(\widetilde{x})~({\rm mod} ~\Ug'\n),
$$
where $\widetilde{x}$ is a representative in $\Ug'$ of the element $x \in V_w=\Ug'/\Ug'\n^w$. The operator $q_w$ does not depend on representation (\ref{w}) of the element $w\in W$.

The image of $q_w$ coincides with the subspace
$$
V^{w}=\{v \in V:~e_\alpha v=0~\forall~\alpha \in \Delta_w \},
$$
and the kernel of $q_w$ is the subspace $\n_w V_w$, where $\n_w$ is the Lie subalgebra of $\g$ generated by the root vectors $e_{ \alpha}$, $\alpha \in \Delta_w$.
\end{proposition}
Operators $q_w$ are called Zhelobenko operators.

Now similarly to the case of extremal projection operators we define classical analogues of operators $q_w$. First, the classical analogue of the module $V_w$ is $\mathbb{C}[\ba^w]'=\mathbb{C}[\ba^w]\otimes_{\mathbb{C}[\h]}\mathbb{C}[\h]'$, the localization of the algebra of regular functions $\mathbb{C}[\ba^w]$ on the Borel subalgebra $\ba^w$ generated by $e_\alpha$, $\alpha \in w\Delta_+$ and by $\h$. Note that $\mathbb{C}[\ba^w]'=\mathbb{C}[\g]'/I^w$, where $I^w$ is the ideal in $\mathbb{C}[\g]'$ generated by the functions $\widehat{e}_{w\alpha}$, $\alpha \in \Delta_+$.

We can state the following proposition which is a classical analogue of Proposition \ref{3}.
\begin{proposition}\label{4}
Let $G$ be a complex semisimple algebraic group with Lie algebra $\g$, $\h$ a Cartan subalgebra of $\g$, $\ba$ the Borel subalgebra of the Lie algebra $\g$ containing $\h$, $\n\subset \ba$ the nilradical of $\ba$, and $N$ the maximal unipotent subgroup $N\subset G$ corresponding to $\n$. Denote by $W$ the Weyl group of the pair $(\g,\h)$.
Let $w\in W$ be an element of the Weyl group $W$, and $\Delta_w=\{ \alpha\in \Delta_+:~w^{-1}\alpha<0 \}$, where $\Delta_+$ is the system of positive roots of the pair $(\g,\ba)$.
Denote by $\n_w$ the Lie subalgebra of $\g$ generated by the root vectors $e_{ \alpha}$, $\alpha \in \Delta_w$.

Let $\mathbb{C}[\ba]'=\mathbb{C}[\ba]\otimes_{\mathbb{C}[\h]}\mathbb{C}[\h]'$ be the localization of the algebra of regular functions $\mathbb{C}[\ba]$ on $\ba$ on $\mathbb{C}[\h]'$, and $\mathbb{C}[\ba^w]'=\mathbb{C}[\ba^w]\otimes_{\mathbb{C}[\h]}\mathbb{C}[\h]'$ the localization of the algebra of regular functions $\mathbb{C}[\ba^w]$ on the Borel subalgebra $\ba^w$ generated by $e_\alpha$, $\alpha \in w\Delta_+$ and by $\h$.

Then the adjoint action map induces a birational equivalence,
\begin{equation}\label{re2}
N_w\times (\h+\n^w_1)\rightarrow \ba,
\end{equation}
where $N_w$ is the subgroup in $G$ corresponding to the Lie subalgebra $\n_w \subset \g$, and $\n^w_1\subset \ba$ is the Lie subalgebra of $\g$ generated by $e_\alpha$, $\alpha\in w(\Delta_+\setminus \Delta_{w^{-1}})\subset \Delta_+$.

Let $\mathbb{C}[\ba]'^{N_w}\subset \mathbb{C}[\ba]'$ be the subspace of elements which are invariant with respect to the action of $N_w$ on $\mathbb{C}[\ba]'$ induced by the adjoint action. Fix a normal ordering $\beta_1, \ldots ,\beta_N$ of $\Delta_+$ such that $w=s_{\beta_k}\ldots s_{\beta_N}$.

Then formula
\begin{equation}\label{Qw}
(Q_wf)(y)=f({\rm Ad} e^{\frac{c_{\beta_{k}}^{N-k}}{\beta_{k}(h)}e_{\beta_{k}}}\ldots {\rm Ad} e^{\frac{c_{\beta_N}}{\beta_N(h)}e_{\beta_N}}y),~y\in \ba',~f\in \mathbb{C}[\ba^w]',
\end{equation}
where the coefficients $c_{\beta_{i}}^{N-i}$, $i=k,\ldots ,N$, $c_{\beta_{N}}^{0}=c_{\beta_{N}}$ are given  as in (\ref{ind}), defines a surjective linear operator $Q_w:\mathbb{C}[\ba^w]'\rightarrow \mathbb{C}[\ba]'^{N_w}$ the kernel of which is generated by the elements $\widehat{e}_{\alpha}\in \mathbb{C}[\ba^w]'$, $\alpha \in \Delta_w$.

The operator $Q_w$ can be expressed by multiplicative formula
\begin{equation}\label{Qw1}
Q_wf=Q_{\beta_N}(\ldots (Q_{\beta_k}(f))\ldots )~{\rm (mod}~I{\rm )},~f\in \mathbb{C}[\ba^w]',
\end{equation}
and operators $Q_\alpha$ are given by
\begin{equation}\label{Qp}
Q_\alpha f=\sum_{n=0}^\infty \frac{(-1)^n}{n!}\widehat{h}_\alpha^{-n}\widehat{e}_{-\alpha}^n \{\widehat{e}_\alpha, \widetilde{f}\}^n,
\end{equation}
where $\widetilde{f}\in \mathbb{C}[\g]'$ is any representative of $f\in \mathbb{C}[\ba^w]'=\mathbb{C}[\g]'/I^w$ in $\mathbb{C}[\g]'$, and
$$
\{e_\alpha \widetilde{f}\}^n=\underbrace{\{\widehat{e}_\alpha,\{\widehat{e}_\alpha, \ldots,\{\widehat{e}_\alpha,\widetilde{f}\}\ldots\}\}}_{\rm n~times}.
$$
\end{proposition}

\begin{proof}
First we show that for $f\in \mathbb{C}[\ba^w]'$ the function $Q_wf\in \mathbb{C}[\ba]'$ given by (\ref{Qw})
is a well--defined $N_w$--invariant function on $\ba'$, where the action of $N_w$ on $\ba'$ is given by the restriction of the adjoint action.

Indeed, by (\ref{ind}) for any $y\in \ba'$ of form (\ref{y}) one can uniquely find coefficients $c_{\beta_{i}}^{N-i}$, $i=k,\ldots ,N$, $c_{\beta_{N}}^{0}=c_{\beta_{N}}$ such that
$$
y^{N-k+1}={\rm Ad} e^{\frac{c_{\beta_{k}}^{N-k}}{\beta_{k}(h)}e_{\beta_{k}}}\ldots {\rm Ad} e^{\frac{c_{\beta_N}}{\beta_N(h)}e_{\beta_N}}y=h+\sum_{\alpha<\beta_{k}}c_\alpha^{N-k+1} e_\alpha.
$$
Since $\Delta_w=\{ \alpha\in \Delta_+:~w^{-1}\alpha<0 \}=\{ \beta_k,\ldots ,\beta_N \}$ this gives a birational equivalence
\begin{equation}\label{re1}
N_w\times (\h+\n^w_1)\rightarrow \ba,
\end{equation}
where $\n^w_1\subset \ba$ is the Lie subalgebra generated by $e_\alpha$, $0<\alpha<\beta_k$.

Now observe that $w(\Delta_+)=w(\Delta_{w^{-1}})\cup w(\Delta_+\setminus \Delta_{w^{-1}})$, and by the definition of the sets $\Delta_w$ $w(\Delta_+\setminus \Delta_{w^{-1}})\subset \Delta_+$. Moreover, by the results of \S3 in \cite{Z3} $w(\Delta_{w^{-1}})=-\Delta_w$.  Observing that the length of $w$ in $W$ is equal to that of $w^{-1}$ we deduce that ${\rm card}~\Delta_w={\rm card}~\Delta_{w^{-1}}$. We also obviously have $w^{-1}w(\Delta_+)=\Delta_+$, and hence $w(\Delta_+)\cap \Delta_w=\emptyset$, and $w(\Delta_+\setminus \Delta_{w^{-1}})=\{\beta_1,\ldots, \beta_{k-1}\}$. Therefore $\n^w_1\subset \ba$ is the Lie subalgebra of $\g$ generated by $e_\alpha$, $\alpha\in w(\Delta_+\setminus \Delta_{w^{-1}})\subset \Delta_+$, and the Borel subalgebra $\ba^w$ is generated by $\h$, by $e_{\beta_1},\ldots, e_{\beta_{k-1}}$, and by $e_\alpha$, $\alpha \in -\Delta_w$. We denote the Lie subalgebra generated by $e_\alpha$, $\alpha \in -\Delta_w$ by $\overline{\n}_w$. This description of $\ba^w$ implies that
\begin{equation}\label{inc}
\h+\n^w_1\subset \ba^w.
\end{equation}

Since $y^{N-k+1}\in \h+\n^w_1$ the right hand side of (\ref{Qw}) is well defined as an element of $\mathbb{C}[\ba]'$ for $f\in \mathbb{C}[\ba^w]'$. Equivalence (\ref{re1}) also implies that $Q_wf$ is $N_w$--invariant. Finally observe that $\ba\cap \ba^w=\h+\n^w_1$ and that $\ba^w=\h+\n^w_1+\overline{\n}_w$ (direct sum of vector spaces), and hence the image of the map $Q_w$ coincides with $\mathbb{C}[\ba]'^{N_w}$, and the kernel of $Q_w$ is generated by the elements $\widehat{e}_{\alpha}\in \mathbb{C}[\ba^w]'$, $\alpha \in \Delta_w$.

Similarly to the case of extremal projection operators one can obtain for the operators $Q_w$ a multiplicative formula analogous to (\ref{qw}),
$$
Q_wf=Q_{\beta_N}(\ldots (Q_{\beta_k}(f))\ldots )~{\rm (mod}~I^w{\rm )},~f\in \mathbb{C}[\ba^w]',
$$
and operators $Q_\alpha$ are given by a formula similar to (\ref{Pp}),
\begin{equation}\label{Qp}
Q_\alpha f=\sum_{n=0}^\infty \frac{(-1)^n}{n!}\widehat{h}_\alpha^{-n}\widehat{e}_{-\alpha}^n \{\widehat{e}_\alpha, \widetilde{f}\}^n,
\end{equation}
where $\widetilde{f}\in \mathbb{C}[\g]'$ is any representative of $f\in \mathbb{C}[\ba^w]'=\mathbb{C}[\g]'/I^w$ in $\mathbb{C}[\g]'$, and
$$
\{e_\alpha \widetilde{f}\}^n=\underbrace{\{\widehat{e}_\alpha,\{\widehat{e}_\alpha, \ldots,\{\widehat{e}_\alpha,\widetilde{f}\}\ldots\}\}}_{\rm n~times}.
$$
\end{proof}

\end{document}